\documentclass{amsart}
\usepackage{amsfonts,amssymb,amsmath,euscript}
\usepackage{graphicx}
\usepackage[matrix,arrow,curve]{xy}
\usepackage{wrapfig}

  \usepackage{amssymb}
\usepackage{mathrsfs}
\usepackage{euscript}
\usepackage{lipsum}

    \newtheorem{thm}{Theorem}                     [section]
    \newtheorem{thm*}{Theorem}
    
    \newtheorem{lemma}[thm]{Lemma}

    \newtheorem{lemma*}{Lemma}    

\newcommand*{\set}[1]{\left\{#1\right\}}           
\newcommand*{\brs}[1]{\left(#1\right)}             
\newcommand{\SqBrs}[1]{\Big[#1\Big]}        
\newcommand{\la}{\langle}

\newcommand{\hilb}{\mathcal H}

\newcommand{\Scal}[1]{\left\langle #1\right\rangle}               

\newcommand{\clK}{\mathcal K}
\newcommand{\clV}{\mathcal V}

\newcommand{\mbR}{\mathbb R}

\newcommand{\ran}{\operatorname{ran}}

\newcommand{\eps}{\varepsilon}

\begin{document}
\title[LAP and singular spectrum]{Limiting absorption principle \\ and singular spectrum}
\author{Nurulla Azamov}
\address{Independent scholar, Adelaide, SA, Australia}
\email{azamovnurulla@gmail.com}
\begin{abstract} 
In this paper I give an explicit construction of an analogue of eigenspace for points of singular spectrum of a self-adjoint operator.
This construction is based on an abstract version of homogeneous Lippmann-Schwinger equation.
\end{abstract}
\maketitle

\bigskip 

\section{Introduction}

\smallskip
Let $\hilb$ and $\clK$ be (complex separable) Hilbert spaces,
$H_0$ a self-adjoint operator on $\hilb$ and $F \colon \hilb \to \clK$ a bounded operator with zero kernel and co-kernel such that 
the sandwiched resolvent 
$$
     T_z(H_0)  = F (H_0-z)^{-1}F^*
$$
is compact. One says that the limiting absorption principle (LAP) holds at a real number $\lambda$ if the norm limit 
\begin{equation} \label{F: T+(0) exists}
     T_{\lambda + i0 }(H_0) :=  \lim_{y \to 0^+}  T_{\lambda + iy }(H_0)  
\end{equation}
exists. Usually LAP means that such limit exists for a.e. $\lambda$ in some open interval, but for the present purpose it suffices to consider it at a single point.
It may well happen that the limit \eqref{F: T+(0) exists} does not exist. In this case there are two scenarios: it is possible that 
the limit
\begin{equation*} 
     T_{\lambda + i0}(H_r), 
\end{equation*}
where 
\begin{equation} \label{F: line}
   H_r = H_0 + rF^*F,
\end{equation}
exists for at least one real number $r \in \mbR,$ or otherwise. In the first scenario the limit exists for all real $r$ except some discrete set of values, called coupling resonance points. 
In the first of these scenarios  $\lambda$ is called a semi-regular point of the pair $H_0,F,$ and in the second $\lambda$ is called essentially singular. 
For a semi-regular point $\lambda$ the kernel, denoted 
\begin{equation} \label{F: Upsilon}
    \Upsilon^1_{\lambda+i0}(H_0),
\end{equation}
of the operator $1 - r T_{\lambda+i0}(H_r) $
is well-defined for non-resonance values of $r$ and in that case it does not depend on the choice of such~$r.$
The aim of this paper is to prove the following theorem. At the end of this introduction I make some remarks explaining why such a theorem is interesting. 

\begin{thm} \label{T1}
Assume the above about $H_0$ and $F.$ Suppose there exists a function~$g$ from 
$L^1(\mbR, (1+x^2)^{-1}\,dx)$ such that for a.e. $\lambda$ 
\begin{equation} \label{F: T1}
    \sup _{y \in (0,1]} \| T_{\lambda + iy} (H_0) \|  \leq g(\lambda).
\end{equation}
Then for a semi-regular point~$\lambda$ 
$$
  \bigcap_{\lambda \in O} \overline{F E_O(H_0)\hilb} = \Upsilon^1_{\lambda+i0}(H_0),
$$
where the intersection is over all open neighbourhoods of~$\lambda.$
\end{thm}

Theorem~\ref{T1} immediately implies the following theorem which gives a partial positive solution to \cite[Subsection 15.9, Conjecture 7]{AzSFIES}.
\begin{thm} \label{T: Conjecture 7 from AzSFIES}
Assume the premise of Theorem \ref{T1}.
If $\chi \in \hilb$ obeys $F\chi \in \Upsilon^1_{\lambda + i0}(H_0)$ then $H_0\chi = \lambda \chi.$
\end{thm}
\noindent 
Thus, given the condition \eqref{F: T1}, Theorems~\ref{T: Conjecture 7 from AzSFIES} and~\cite[Theorem 4.1.1] {AzSFIES} assert that 
\begin{equation} \label{F5}
   F\clV(\lambda, H_0) = \Upsilon^1_{\lambda+i0}(H_0) \cap \ran (F),
\end{equation}
where $\clV(\lambda, H_0)$ is the eigenspace of $H_0$ corresponding to en eigenvalue $\lambda.$

\smallskip
{\bf Remark 1.}
Assuming that the rigging $F$ is bounded, 
the vector space $\ran (F)$ endowed with the graph-norm, is a Hilbert space naturally isomorphic to~$\hilb,$ the isomorphism being given by~$F$ itself.
The equality \eqref{F5} asserts that the eigenvectors of $H_0$ corresponding to an eigenvalue $\lambda$ can be interpreted as those elements of 
$\Upsilon^1_{\lambda+i0}(H_0)$ which belong to the image of~$F.$
Vectors from $\Upsilon^1_{\lambda+i0}(H_0)$ which do not belong to $\ran(F)$ can therefore be interpreted as generalised eigenvectors of~$H_0.$
Moreover, these generalised eigenvectors are  $F$-images of elements of the singular subspace of~$H_0.$

\smallskip
{\bf Remark 2.} 
Instead of the straight line \eqref{F: line} we could have worked with the line 
$
   H_r = H_0 + rF^*JF,
$   
where $J$ is any bounded self-adjoint operator on $\clK$ such that for some $r \in \mbR$ the limit
\begin{equation} \label{F:T+(J)}
   T_{\lambda+i0} (H_0 + rF^*JF)
\end{equation}
exists (such operators $J$ are called regular directions), --- proof is exactly the same.
But as far as Theorem \ref{T1} is concerned, this makes no difference since the solution set \eqref{F: Upsilon} to the equation 
\begin{equation} \label{F: Lippmann-Schwinger}
   (1 - r T_{\lambda + i0} (H_0 + rF^*JF) J ) u = 0
\end{equation}
does not depend on a choice of a regular direction $J$ and a non-resonant value of $r,$ see \cite{AzSFIES}.
Another reason for considering the direction $\mathrm{Id}$ instead of an arbitrary $J$ is that if the limit \eqref{F:T+(J)} exists for some bounded $J$ then it also exists for the identity operator $J = \mathrm{Id},$
see \cite{AzSFIESIV}.

\smallskip
{\bf Remark 3.} 
The equation \eqref{F: Lippmann-Schwinger} is nothing else but the homogeneous version of an abstract Lippmann-Schwinger equation,
see e.g. \cite[\S 4.3]{BeShu} or \cite{TayST}.
 For a semi-regular energy $\lambda,$ the limit $T_{\lambda+i0}(H_0)$ fails to exist if and only if the 
equation \eqref{F: Lippmann-Schwinger} has a non-zero solution. The solutions can be interpreted as bound states or meta-stable states (also called resonances) of $H_0$ with energy~$\lambda,$ 
where bound states  correspond to elements of \eqref{F5}. 

\smallskip
{\bf Remark 4.}
Theorem \ref{T1} is not unrelated to the well-known Simon-Wolff criterion \cite{SW}, see also \cite{SimTrId2}. 
This relation will soon be discussed elsewhere.

\smallskip
{\bf Remark 5.}
It is not essential to assume that the rigging operator $F$ is bounded. I made this assumption to avoid more technical details. 



%
%

\section{Proof of Theorem \ref{T1}}

\begin{lemma} Let $H_1 = H_0 + V.$ For any $w \in \rho(H_0)$ and $z \in \rho(H_1)$ 
\begin{equation} \label{F: tricky equality}
  (w-z)R_{w}(H_0) R_{z}(H_1) = - R_{z}(H_1) + R_{w}(H_0) \SqBrs{1 -  V R_{z}(H_1)}.
\end{equation}
\end{lemma}
\begin{proof}
    Using the second resolvent identity
$$
  R_{w}(H_0) = (1 - R_{w}(H_1)V)^{-1} R_{w}(H_1)
$$
we rewrite $R_{w}(H_0)$ in terms of $R_{w}(H_1)$ with the aim to use next the first resolvent identity:
\begin{equation*}
  \begin{split}
     (w-z)R_{w}(H_0) R_{z}(H_1) & = (w-z)(1 - R_{w}(H_1)V)^{-1} R_{w}(H_1) R_{z}(H_1) \\
         & = (1 - R_{w}(H_1)V)^{-1} \SqBrs{R_{w}(H_1) - R_{z}(H_1)}\\
         & = R_{w}(H_0) - (1 - R_{w}(H_1)V)^{-1}R_{z}(H_1).
  \end{split}
\end{equation*}
Since $(1 - R_w(H_1)V)^{-1} = 1 + R_w(H_0)V,$ this gives 
\begin{equation*}
  \begin{split}
      (w-z)R_{w}(H_0) R_{z}(H_1) &  = R_{w}(H_0) - (1 + R_w(H_0)V)R_{z}(H_1) \\
       & = - R_{z}(H_1)+ R_{w}(H_0)\SqBrs{1 - V R_{z}(H_1)}.
  \end{split}
\end{equation*}

\end{proof}

We only need to prove the inclusion
\begin{equation} \label{F: Upsilon subset}
      \Upsilon^1_{\lambda+i0}(H_0)  \subset \bigcap_{\lambda \in O} \overline{F E_O(H_0)\hilb},
\end{equation}
since the other inclusion was proved in \cite{AzSFIESVI}.
Let $u \in \Upsilon_{\lambda+i0}^1(H_0),$ that is, 
\begin{equation} \label{F: HLSE}
    (1 - s T_{\lambda+i0} (H_s)) u = 0.
\end{equation}
Since solutions of (\ref{F: HLSE}) do not depend on the choice of a non-resonant value of~$s,$ without loss of generality
we can assume that $s=1,$ in particular assuming that this value is non-resonant.  
Let
$$
  f_{\lambda+iy} := R_{\lambda+iy}(H_1) F^* u.
$$
Our aim is to show that for small enough $y>0$
the spectral representation of the vector $f_{\lambda+iy}$ with respect to $H_0$ is concentrated near~$\lambda.$ 

\medskip

%

\begin{lemma} For $\lambda, \ x \in \mbR$ and $y>0$ we have 
\begin{equation} \label{F: Im (fy,R(mu) fy)}
  \begin{split}
      \Im  \langle f_{\lambda+iy}, & \ R_{x+iy}(H_0) f_{\lambda+iy} \rangle  \\
           & = (x-\lambda)^{-1}
           \Im \brs{(x-\lambda-2iy)^{-1}\Scal{[...] u,  \SqBrs{u - T_{\lambda+iy}(H_1)u}}},
  \end{split}
\end{equation}
where 
$$
  [...] = -T_{\lambda+iy}(H_1) + T_{x-iy}(H_0) [ 1 -    T_{\lambda+iy}(H_1)].
$$

\end{lemma}
\begin{proof} 
Using (\ref{F: tricky equality}), we have 
\begin{equation} \label{F: ()Rf=-f+...}
  \begin{split}
      (x - \lambda ) R_{x+iy}(H_0) f_{\lambda+iy} 
         & = (x - \lambda)R_{x+iy}(H_0) R_{\lambda+iy}(H_1)F^*u
      \\ & = \brs{-R_{\lambda+iy}(H_1)  +  R_{x+iy}(H_0)\SqBrs{1 -  F^*F R_{\lambda+iy}(H_1)}} F^*u
      \\ & = -f_{\lambda+iy} + R_{x+iy}(H_0)F^*\SqBrs{u - T_{\lambda+iy}(H_1)u}.
  \end{split}
\end{equation}
Taking the scalar product of both sides of (\ref{F: ()Rf=-f+...}) 
with $\la f_{\lambda+iy}|$ and then taking the imaginary part of the resulting scalar products 
we get
\begin{equation} \label{F: abc}
  (x-\lambda)  \Im \Scal{f_{\lambda+iy}, R_{x+iy}(H_0) f_{\lambda+iy}} 
    = \Im \Scal{R_{x-iy}(H_0)f_{\lambda+iy}, F^* \SqBrs{u - T_{\lambda+iy}(H_1)u}}.
\end{equation}
Using  \eqref{F: tricky equality} again,
we transform the first argument of the last scalar product
as follows:
\begin{equation} \label{F: R(y)f=}
  \begin{split}
      R_{x-iy}(H_0)f_{\lambda+iy} & = R_{x-iy}(H_0)R_{\lambda+iy}(H_1)F^*u
    \\ & = (x-\lambda - 2iy)^{-1}
    \\ & \mbox{ } \hskip 0.7 cm \times \SqBrs{-R_{\lambda+iy}(H_1) + R_{x-iy}(H_0) \big(
              1 - F^*FR_{\lambda+iy}(H_1)\big) } F^*u.
  \end{split}
\end{equation}
Hence, denoting by $[...]$ the expression in the last pair of square brackets, we get
from (\ref{F: abc})
\begin{equation*}
  \begin{split}
    (x-\lambda) \Im & \Scal{f_{\lambda+iy}, R_{x+iy}(H_0) f_{\lambda+iy}} 
   \\ &= 
         \Im \brs{(x-\lambda - 2iy)^{-1}\Scal{F[...]F^* u,  \SqBrs{u - T_{\lambda+iy}(H_1)u}}},
  \end{split}
\end{equation*}
as required. 
\end{proof}

\begin{lemma}  Under the premise of Theorem \ref{T1},
for any $\delta > 0$ 
\begin{equation} \label{F: lim int (f, R f)= 0}
    \lim_{ y \to 0^+}   \int_{\mbR \setminus (\lambda-\delta,\lambda+\delta)} \Im \Scal{f_{\lambda+iy}, R_{x+iy}(H_0) f_{\lambda+iy}}\,dx = 0.
\end{equation}
\end{lemma}
\begin{proof}
We will use \eqref{F: Im (fy,R(mu) fy)} for the integrand.
The contribution of the summand   $-T_{\lambda+iy}(H_1)$ in $[\ldots]$ to the limit \eqref{F: lim int (f, R f)= 0}  is clearly zero.
Thus, introducing the notation 
$$
     \chi_{\lambda + iy} :=  \SqBrs{1 - T_{\lambda+iy}(H_1)} u,
$$
it suffices to prove that the limit of the integral of 
\smallskip
$$
    (x-\lambda)^{-1}
            \Im \brs{(x-\lambda-2iy)^{-1}\Scal{ T_{x-iy}(H_0) \chi_{\lambda + iy}, \chi_{\lambda + iy} } }
$$
\smallskip
\noindent 
over $x \notin (\lambda - \delta, \lambda + \delta)$ 
goes to zero. 

By \eqref{F: HLSE}, the vector $\chi_{\lambda + iy}$ converges to zero as $y \to 0^+.$
Thus, by the assumption \eqref{F: T1}, the integrand converges to zero for a.e. $x \notin (\lambda - \delta, \lambda + \delta).$ 
Moreover, by the same assumption we can apply the Lebesgue Dominated Convergence Theorem to interchange the limit $y \to 0^+$ with the integration. 
\end{proof} 

Now we can complete proof of Theorem \ref{T1}.
By Stone's formula, 
the integral of $\pi^{-1}\Im R_{x + iy}(H_0)$ over the complement of $(\lambda-\delta,\lambda+\delta)$ converges strongly to 
$$
     E_{\mbR \setminus (\lambda-\delta,\lambda+\delta)}(H_0) + \frac 12 {E_{\set{\lambda-\eps,\lambda+\eps}}},
$$ 
as $y \to 0^+.$
Combining this with \eqref{F: lim int (f, R f)= 0} gives
\begin{equation} \label{F: E y to 0}
   E_{\mbR \setminus (\lambda-\delta,\lambda+\delta)}(H_0)f_{\lambda+iy} \to 0
\end{equation}
as $y\to 0^+.$ 

In order to prove the inclusion \eqref{F: Upsilon subset},
it suffices to show that for any $\eps>0$ and $\delta>0$ there 
exists $\psi \in E_{(\lambda-\delta,\lambda+\delta)}(H_0)$ such that 
the distance between $u$ and $F\psi$ is less than $\eps.$ 
We claim that for small enough $y>0$ the choise
$$
    \psi  = E_{(\lambda - \delta, \lambda + \delta)} f_{\lambda+iy}
$$
works. Indeed,
\begin{equation*}
  \begin{split}
      \| u - F \psi \| & =   \| u -  F E_{(\lambda - \delta, \lambda + \delta)} f_{\lambda+iy} \|  \\
         & \leq \| u - F f_{\lambda+iy} \| + \|  F f_{\lambda+iy}  -  F E_{(\lambda - \delta, \lambda + \delta)} f_{\lambda+iy} \|   \\
         & \leq \| u - T_{\lambda+iy}(H_1) u \| + \|  F\| \| f_{\lambda+iy}  -   E_{(\lambda - \delta, \lambda + \delta)} f_{\lambda+iy} \|.
  \end{split}
\end{equation*}
Since $u$ is a solution to \eqref{F: HLSE} (with $s=1$), for all small enough $y>0$ the first summand is $< \eps/2.$
By \eqref{F: E y to 0}, for all small enough $y>0$ the second summand is also $< \eps/2.$

Proof is complete.

\bigskip\bigskip 
{\it Acknowledgements.} I thank my wife for financial support during the work on this paper.

\end{document}